\documentclass[12pt,twoside,final,psamsfonts]{amsart}

\usepackage[psamsfonts]{amssymb}
\usepackage{times,a4wide}
\usepackage{MnSymbol}

\usepackage{color}		
\usepackage{epsfig}

\usepackage{graphicx}

\newcommand{\nc}{\newcommand}

\numberwithin{equation}{section}
\theoremstyle{plain}
\newtheorem{proposition}[equation]{Proposition}

\newtheorem{theorem}[equation]{Theorem}
\newtheorem{lemma}[equation]{Lemma}
\theoremstyle{definition}

\newtheorem{example}[equation]{Example}
\newtheorem{remark}[equation]{Remark}

\newcommand{\N}{{\mathbb N}}
\newcommand{\D}{{\mathbb D}}

\newcommand{\T}{{\mathbb T}} 

\newcommand{\Z}{{\mathbb Z}}

\newcommand{\lra}{\longrightarrow}

\newcommand{\eps}{\varepsilon}
\newcommand{\vp}{\varphi}
\nc{\bea}{\begin{eqnarray}}
\nc{\eea}{\end{eqnarray}}
\nc{\beqa}{\begin{eqnarray*}}
\nc{\eeqa}{\end{eqnarray*}}
\nc{\Hi}{H^{\infty}}
\nc{\loi}{\ell^{\infty}}
\nc{\NL}{N^+\vert \Lambda}
\nc{\hf}{{\mathcal H}_{\phi}}
\nc{\liL}{\lambda\in\Lambda}
\nc{\nn}{\nonumber}
\nc{\dst}{\displaystyle}

\newenvironment{proof*}{\vskip 2mm\noindent {}}{$\blacksquare$ \vskip 2mm}
\numberwithin{equation}{section}

\renewcommand{\Im}{\mbox{Im}}

\title{Bad boundary behavior in star invariant subspaces I}

\author{Andreas Hartmann \& William T.\ Ross}

\address{Institut de Math\'ematiques de Bordeaux,
Universit\'e Bordeaux I, 351 cours de la Lib\'eration,
33405 Talence, France}

\address{Department of Mathematics and Computer Science, University of Richmond, VA 23173, USA}

\email{hartmann@math.u-bordeaux.fr, wross@richmond.edu}

\date{\today}

\keywords{Hardy spaces, star invariant subspaces, non-tangential limits, Blaschke products}

\subjclass{30J10, 30A12, 30A08}

\begin{document}

\bibliographystyle{amsalpha}

\begin{abstract} We discuss the boundary behavior of functions in
star invariant subspaces $(B H^2)^{\perp}$, where $B$ is a Blaschke product. 
Extending some results of
Ahern and Clark, we are particularly interested
in the  growth rates of functions at points of the spectrum of $B$ where $B$ 
does not admit a 
derivative in the sense of Carath\'eodory.  
\end{abstract}

\maketitle


\section{Introduction}

For a Blaschke product $B$
with zeros $(\lambda_{n})_{n \geq 1} \subset \D = \{|z| < 1\}$, repeated according to multiplicity, let us recall the following theorem of Ahern and Clark \cite{AC70} about the ``good'' non-tangential boundary behavior of functions in the model spaces $(B H^2)^{\perp} := H^2 \ominus B H^2$ \cite{Niktr} of the Hardy space $H^2$ of $\mathbb{D}$ \cite{Duren, Garnett}. 

\begin{theorem}[\cite{AC70}]  \label{AC-paper}
For a Blaschke product $B$ with zeros $(\lambda_{n})_{n \geq 1}$ and $\zeta \in \T := \partial \mathbb{D}$, the following are equivalent: 
\begin{enumerate}
\item Every $f \in (B H^2)^{\perp}$ has a non-tangential limit at $\zeta$, i.e., 
$$f(\zeta) := \measuredangle \lim_{\lambda \to \zeta} f(\lambda) \; \; \mbox{exists.}$$
\item $B$ has an angular derivative in the sense of Carath\'eodory 
at $\zeta$, i.e., 
$$\measuredangle \lim_{z \to \zeta} B(z) = \eta \in \T \quad \mbox{and ${\displaystyle \measuredangle \lim_{z \to \zeta} B'(z)}$ exists.}$$
\item  The following condition holds
\begin{equation} \label{AC1}
\sum_{n \geq 1} \frac{1 - |\lambda_{n}|}{|\zeta - \lambda_{n}|^2} < \infty.
\end{equation}
\item The family of reproducing kernels for $(B H^2)^{\perp}$
$$k^B_{\lambda}(z) := \frac{1 - \overline{B(\lambda)} B(z)}{1 - \overline{\lambda} z}$$ is uniformly norm bounded in each fixed Stolz domain 
$$\Gamma_{\alpha,\zeta} :=\left\{z \in \D: 
\frac{|z - \zeta|}{1 - |z|} < \alpha \right\}, \quad \alpha \in (1, \infty).$$
\end{enumerate}
\end{theorem}

We point out three things here. First, the equivalence of conditions (2) and (3) of this theorem is a classical result of Frostman \cite{Frost}. Second, this theorem can be extended to characterize the existence of non-tangential boundary 
limits of the derivatives (up to a given order) of functions in $(B H^2)^{\perp}$ as well as the boundary behavior of functions in $(I H^2)^{\perp}$ where $I$ is a general  inner function \cite{AC70}. Third, there is a version of this result for various types of \emph{tangential} boundary behavior of $(B H^2)^{\perp}$ functions \cite{Cargo, Protas}.
Of course there is the well-known result 
(see e.g.\ \cite[p.~78]{Niktr}) which says that every $f \in (B H^2)^{\perp}$ has an analytic continuation across the complement of the accumulation points of the zeros of $B$. 

In this paper we consider the growth of functions in $(B H^2)^{\perp}$ at the points $\zeta \in \T$ where \eqref{AC1} fails. Thus, as in the title of this paper, we are looking at the ``bad'' boundary behavior of functions from $(B H^2)^{\perp}$.  
First observe that every function $f \in H^2$ satisfies 
\begin{equation} \label{H2growth}
|f(\lambda)| = o\left(\frac{1}{\sqrt{1 - |\lambda|}}\right), \quad \lambda \in \Gamma_{\alpha, \zeta},
\end{equation}
and this growth is, in a sense,  maximal. 
As seen in the Ahern-Clark theorem, functions in $(B H^2)^{\perp}$ can be significantly better behaved depending on the distribution of the zeros of $B$. We are interested in examining Blaschke products for which the growth rates for functions in $(B H^2)^{\perp}$ are somewhere between the Ahern-Clark situation, where every function has a non-tangential limit, and the maximal allowable growth in \eqref{H2growth}.

To explain this a bit more, let $\zeta=1$ and observe that
\begin{equation} \label{trivial-est}
 |f(\lambda)|=|\langle f,k^B_{\lambda}\rangle|\le
 \|f\|\left(\frac{1-|B(\lambda)|^2}{1-|\lambda|^2}\right)^{1/2}, \quad f \in (B H^2)^{\perp}, \lambda \in \D.
\end{equation}
In the above, $\|\cdot\|$ denotes the usual norm in $H^2$. 
So, in order to give an upper estimate of the admissible
growth in a Stolz domain $\Gamma_{\alpha, 1}$, we have to control $\|k^B_{\lambda}\|$ 
 which ultimately involves getting a handle on  how
fast $|B(\lambda)|$ goes to 1 in $\Gamma_{\alpha, 1}$.

Of course the subtlety occurs  when
$$\measuredangle \lim_{z \to 1} B(z) = \eta \in \T$$ which is implied by the Frostman condition \cite{CL, Frost}
\begin{equation} \label{Frost}
 \sum_{n\ge 1}\frac{1-|\lambda_n|}{|1-\lambda_n|}<\infty.
\end{equation}
Observe the power in the denominator in \eqref{Frost} with respect to the Ahern-Clark
condition \eqref{AC1}.

The main results of this paper will be non-tangential growth estimates of functions in $(B H^2)^{\perp}$ via non-tangential growth estimates of the norms of the kernel functions. Our main results (Theorem \ref{C-upper-growth}, Theorem \ref{Oricycle-seq}, and Theorem \ref{thm3.1})  will be estimates of the form 
$$\|k^{B}_{r}\| \asymp h(r), \quad r \to 1^{-},$$
for some $h: [0, 1) \to \mathbb{R}_{+}$ which depends on the position of the zeros of the Blaschke product $B$ near $1$. This will, of course via \eqref{trivial-est}, yield the estimate 
$$|f(r)| \lesssim h(r), \quad f \in (B H^2)^{\perp}, \quad r \to 1^{-}.$$ To get a handle on the sharpness of this growth estimate, we will show (Theorem \ref{thm4.2}) that for every $\varepsilon > 0$, there exists an $f \in (B H^2)^{\perp}$ satisfying
\begin{equation} \label{A-opt}
|f(r)| \gtrsim \frac{h(r)}{\log^{1+\eps} h(r)}, \quad r \to 1^{-}.
\end{equation}

While this estimate might not be optimal, it allows to show
that a certain sequence of reproducing kernels cannot form 
an  unconditional sequence  (see 
Section \ref{section5}).

Though a general result will be discussed in Section 4, the two basic types of Blaschke sequences $(\lambda_n)_{n \geq 1}$ for which we can get concise estimates of $\|k^{B}_{r}\|$, are 
\begin{equation} \label{sample-1}
\lambda_{n} = (1 - x_n 2^{-2n}) e^{i 2^{-n}}, \quad x_{n} \downarrow 0,
\end{equation}
which approaches $1$ very tangentially, and 
\begin{equation} \label{sample-2}
\lambda_{n} = (1 - \theta_{n}^{2}) e^{i \theta_n}, \quad 0 < \theta_n < 1, \quad \sum_{n \geq 1} \theta_n < \infty,
\end{equation}
which approaches $1$ along an oricycle. For example,  when $x_n = 1/n$ in \eqref{sample-1}, we have the upper estimate (see Example \eqref{est-loglog}(1))
$$|f(r)| \lesssim\sqrt{ \log \log \frac{1}{1 - r}}, \quad r \to 1^{-},$$
for all $f \in (B H^2)^{\perp}$.
This estimate is optimal in the sense of \eqref{A-opt}.

Picking $\theta_{n} = 1/n^{\alpha}, \alpha > 1$, in \eqref{sample-2}, we have the estimate (see Example \eqref{Example4.32}(1))
$$|f(r)| \lesssim \frac{1}{(1 - r)^{\frac{1}{2 \alpha}}}, \quad r \to 1^{-}.$$
Compare these two results to the growth rate in \eqref{H2growth} of a generic $H^2$ function. 


This is the first of two papers on ``bad'' boundary behavior of $(I H^2)^{\perp}$ ($I$ inner) functions near a fixed point on the circle. In this paper we consider the case when $I$ 
is a Blaschke product giving exact estimates on the norm of the reproducing
kernel. The next paper will consider the case when $I$ is a general inner function
providing only upper estimates.


\section{What can be expected}\label{section2}

 

 We have already mentioned
that every $f \in H^2$ satisfies
\begin{equation}\label{little-oh-2}
  |f(\lambda)|=o\left(\frac{1}{\sqrt{1-|\lambda|}}\right),
 \quad \lambda\in\Gamma_{\alpha, \zeta}.
\end{equation}
%

The little-oh condition in \eqref{little-oh-2}
is, in a sense,  sharp since one can construct suitable outer functions whose non-tangential growth gets arbitrarily   close to \eqref{little-oh-2}.

Contrast this with the following result which shows that functions in certain $(B H^2)^{\perp}$ spaces can {\em not} reach the maximal growth in \eqref{little-oh-2}. Recall that a sequence
$\Lambda= (\lambda_n)_{n \geq 1}\subset \D$ is \emph{interpolating} if 
 $H^2|\Lambda =\{(a_n)_{n \geq 1}:\sum_n (1-|\lambda_n|^2)
|a_n|^2<\infty\}$.

\begin{proposition}[\cite{SS}]\label{prop2.9}
Let $B$ be a Blaschke product whose zeros $\lambda_n$ form an interpolating
sequence
and tend non-tangentially to $1$.
Then
\[
  |f(\lambda_n)|=\eps_n\frac{1}{\sqrt{1-|\lambda_n|}},\quad \forall n\in\N,
\]
for $f\in (B H^2)^{\perp}$
if and only if $(\eps_n)_{n \geq 1} \in \ell^2$.
\end{proposition}

Strictly speaking this result is stated in $H^2$ (and for arbitrary interpolating
sequences), but since 
functions in $BH^2$ vanish on $\Lambda$,
we obviously have $(BH^2)^{\perp}|\Lambda=H^2|\Lambda$.

A central result in our discussion is the following lemma.

\begin{lemma}\label{KeyLemma}
If $B$ is a Blaschke product with zeros $\lambda_n = r_n e^{i \theta_n}$ and $\measuredangle \lim_{z \to 1} B(z) = \eta \in \T$,  then
$$\|k^B_{r}\|^2 \asymp \sum_{n \geq 1} \frac{1 - r_{n}^{2}}{|1- \overline{\lambda}_nr |^2}, \quad r \in (0, 1).$$
\end{lemma}

(The estimate extends naturally to a Stolz angle.)


\begin{proof}
Since $\measuredangle \lim_{z \to 1} B(z) = \eta \in \T$, the zeros of $B$ (after some point) can not lie in $\Gamma_{\alpha, 1}$. Thus if 
$$b_{\lambda}(z) = \frac{z - \lambda}{1 - \overline{\lambda} z},$$ then $\inf_{n \geq 1} |b_{\lambda_n}(r)| \geq \delta > 0$ and so
$$\log \frac{1}{|b_{\lambda}(r)|^2} \asymp 1 - |b_{\lambda_n}(r)|^2.$$
Use the well-known identity
$$1 - |b_{\lambda_n}(r)|^2 = \frac{(1 - r^2)(1 - |\lambda_{n}|^2)}{|1 - r \overline{\lambda_n} |^2},
$$
to get
\beqa
 \log|B(r)|^{-2}
 &=&\sum_{n \geq 1}\log\frac{1}{|b_{\lambda_n}(z)|^2}
 \asymp \sum_{n \geq 1}\frac{(1-|\lambda_n|^2)(1-|r|^2)}{|1-\overline{\lambda_n} r|^2}\\
 &\asymp&
 (1-r^2)\sum_{n \geq 1}\frac{(1-r_n^2)}{|1-\overline{\lambda_n} r|^2}.
\eeqa
Since $|B(r)|\to 1$ when $r\to 1^{-}$ the latter quantity goes to $0$ and so
\[
 \|k^B_r\|^2=\frac{1-|B(r)|^2}{1-r^2}
 \asymp  -\frac{\log|B(r)|^2}{1-r^2}\\
 \asymp\sum_{n \geq 1}\frac{1-r_n^2}{|1-\overline{\lambda_n} r|^2} \qedhere
\] 
\end{proof}

\section{Key examples}\label{S3nn}

We will prove a general growth result in Theorem \ref{thm3.1}. But just to give a more tangible approach to the subject, let us begin by obtaining growth estimates of functions in $(B H^2)^{\perp}$ for Blaschke products $B$ whose zeros are 
$$\lambda_{n} = (1 - x_n 2^{-2n}) e^{i 2^{-n}}, \quad x_{n} \downarrow 0,$$
which approaches $1$ very tangentially, or 
$$\lambda_{n} = (1 - \theta_{n}^{2}) e^{i \theta_n}, \quad 0 < \theta_n < 1, \quad \sum_{n \geq 1} \theta_n < \infty,$$
which (essentially) approaches $1$ along an oricycle. 

\noindent {\bf First class of examples:} $\Lambda=(\lambda_k)_{k \geq 1}$ with
$\lambda_k=r_ke^{i\theta_k}$ and
\bea\label{defseq}
 1-r_k=x_k\theta_k^2, \quad \theta_k=\frac{1}{2^k}, \quad k\in \N. 
\eea
Since $x_k \downarrow 0$,  the sequence $\Lambda$ goes tangentially to 1.
The faster $x_{k}$ decreases to zero, the more tangential the sequence $\Lambda$.
This also implies that
\[
 \sum_{n\ge 1}(1-|\lambda_n|)=\sum_{n\ge 1}(1-r_n)=\sum_{n\ge 1}
 \frac{x_n}{2^{2n}}<\infty,
\]
and so $\Lambda$ is indeed a Blaschke sequence.

We will need the well known Pythagorean type result:
if $\lambda = r e^{i \theta}$, $r \in (0, 1)$, $\rho \in (0, 1]$, then
\bea\label{Pythag-Lemma}
|1 - \rho \lambda|^2 \asymp
  (1 - \rho r)^2 + \theta^2\asymp  ((1 - \rho r)+ \theta)^2, 
 \quad \rho \approx 1, r \approx 1, \theta \approx 0.
\eea

Observe that using \eqref{Pythag-Lemma} we get
\[
 |1-\lambda_k| \asymp \sqrt{(1-r_k)^2+\theta_k^2} 
  \asymp (1-r_k)+\theta_k 
  \asymp  x_k\theta_k^2+\theta_k
  \asymp \theta_k.
\]
Hence 
\[
\sum_{n \geq 1} \frac{1 - |\lambda_{n}|}{|1 - \lambda_{n}|} 
 \asymp\sum_{n\ge 1}\frac{1-r_n}{\theta_k}=\sum_{n\ge 1} \theta_n x_n < \infty
\]
and so condition \eqref{Frost} is satisfied thus ensuring 
$\measuredangle \lim_{z \to 1} B(z) = \eta \in \T$. 
Similarly, 
\[
\sum_{n \geq 1} \frac{1 - |\lambda_{n}|}{|1 - \lambda_{n}|^2} 
 \asymp \sum_{n\ge 1}x_n.
\]
So, in light of the Ahern-Clark result \eqref{AC1}, we
will be interested in the ``bad behavior'' scenario when $\sum_{n\ge 1} x_n=+\infty$.


\begin{theorem} \label{C-upper-growth}
Consider the Blaschke product whose zeros are 
$$\lambda_{n} = (1 - x_n 2^{-2n}) e^{i 2^{-n}}, \quad 
x_{n}\downarrow 0.$$
Set $$\sigma_{N} := \sum_{n = 1}^{N} x_n,$$
and let $\vp_0$ be the piecewise affine function with 
$\vp_0(N) = \sigma_N$, and let $\vp$ be defined by
$$\vp(y) := \vp_0\left(\log_2 \frac{1}{1 - y}\right).$$
Then 
$$
 \|k_z^B\|\asymp \sqrt{\vp(|z|)},\quad z\in \Gamma_{\alpha,1},
$$
and so every $f \in (B H^2)^{\perp}$ satisfies 
$$|f(z)| \lesssim \sqrt{\phi(|z|)}, \quad z \in \Gamma_{\alpha, 1}.$$
\end{theorem}

Note that $\vp_0$ is actually a concave function.

Before discussing the proof, here are two concrete
examples showing how the growth slows down when approaching
the Ahern-Clark situation, i.e., the summability of the sequence $(x_n)_{n \geq 1}$.

\begin{example} \label{est-loglog}
\begin{enumerate}
\item If $B$ is a Blaschke product whose zeros are 
$$\lambda_{n} = (1 - x_n 2^{-2n}) e^{i 2^{-n}}, \quad x_n = \frac{1}{n},$$
 then $$\sigma_{N} = \sum_{n = 1}^{N} \frac{1}{n} \asymp \log N$$
and so every  $f \in (B H^2)^{\perp}$ satisfies the growth condition 
$$|f(r)| \lesssim \sqrt{\log \log \frac{1}{1 - r}}, \quad r \to 1^{-}.$$
\item If the zeros of $B$ are 
$$\lambda_{n} = (1 - x_n 2^{-2n}) e^{i 2^{-n}}, \quad x_n = \frac{1}{n \log n},$$
then 
$\sigma_{N} \asymp \log \log N$ and so every $f \in (B H^2)^{\perp}$ satisfies 
$$|f(r)| \lesssim\sqrt{ \log \log \log \frac{1}{1 - r}}, \quad r \to 1^{-}.$$
\end{enumerate}
\end{example}

\begin{proof}[Proof of Theorem \ref{C-upper-growth}]
Set $\rho_N=1-2^{-N}$ and $\theta_k=2^{-k}$.
Using \eqref{Pythag-Lemma} 
we have
\beqa
 |1-\rho_N\lambda_k|^2 
 &\asymp& (\theta_k+(1-\rho_Nr_k))^2
 =(\theta_k+(1-\rho_N(1-x_k\theta_k^2)))^2\\
 &=&(\theta_k+(1-\rho_N)+\rho_Nx_k\theta_k^2)))^2,
\eeqa
and, by our assumption $x_k\theta_k^2 \ll \theta_k$ when $k\to\infty$, we get
\bea\label{eq1p4}
 |1-\rho_N\lambda_k|^2\asymp  (\theta_k+(1-\rho_N))^2,
\eea
Hence
\bea\label{estimquot}
 \frac{1-r_k^2}{|1-\rho_N\lambda_k|^2}
 &\asymp& \frac{x_k\theta_k^2}{(\theta_k+(1-\rho_N))^2} 
 =\frac{x_k\theta_k^2}{(\theta_k+\theta_N)^2}\asymp
 \left\{\begin{array}{ll}
	\frac{\dst x_k\theta_k^2}{\dst \theta_k^2} & \text{if }k\le N\\
	\frac{\dst x_k\theta_k^2}{\dst \theta_N^2} & \text{if }k > N
 \end{array}
 \right.\nn\\
  &\asymp& \left\{\begin{array}{ll}
	 x_k & \text{if }k\le N,\\
	\frac{\dst x_k\theta_k^2}{\dst \theta_N^2} & \text{if }k > N.
 \end{array}
 \right.
\eea
Thus we can split the sum in Lemma \ref{KeyLemma} into two parts
\[
 \|k_{\rho_N}^B\|^2\asymp
 \sum_{k\ge 0}\frac{1-r_k^2}{|1-\rho_N\lambda_k|^2}
 \asymp \sum_{k\le N}x_k
 +2^{2N}\sum_{k\ge N+1}{x_k\theta_k^2}.
\]
The first term is exactly $\sigma_N$ while the second is bounded by
a uniform constant (recall that we are assuming $x_n\downarrow 0$
and $\theta_k = 2^{-k}$) and hence negligible with respect to $\sigma_N$ which we assume increases to infinity.
This immediately gives us
the required estimate for $\rho_N=1-1/2^N$:
\[
 \|k^B_{\rho_N}\|^2 \asymp \sigma_N
 =\vp_0(N)=\vp(\rho_N).
\]
In order to get the same estimate for $z\in \Gamma_{\alpha, 1}$ we need
the following well known result: 
\begin{equation} \label{tech-est}
|b_{\lambda}(\mu)|\le \eps<1 \Rightarrow
\frac{1-\eps}{1+\eps}\le
 \frac{|1-\overline{\lambda}z|}{|1-\overline{\mu}z|}
 \le \frac{1+\eps}{1-\eps}, \quad z \in \D.
 \end{equation}
 
Now let 
$z\in \Gamma_{\alpha, 1}$ and suppose that $|z|>1/2$.
Then there exists an $N$ such that $$|b_z(\rho_N)|=|b_z(1-2^{-N})|\le \delta < 1$$ (where
$\delta$ only depends on the opening of the Stolz angle). 
Hence 
\begin{equation} \label{QQQ}
 \|k^B_z\|^2\asymp\sum_{n \geq 1}\frac{1-r_n^2}{|1-\overline{\lambda}_nz|^2}
 \asymp  \sum_{n \geq 1}\frac{1-r_n^2}{|1-\overline{\lambda}_n\rho_N|^2}
 \asymp\|k_{\rho_N}\|^2,
\end{equation}
and so
\[
 \|k^B_z\|^2 \asymp\|k^B_{\rho_N}\|^2 \asymp\sigma_N.
\]  
Since $x_n\downarrow 0$ we have 
$\sigma_N\asymp\sigma_{N+1}\asymp\sigma_{N-1}$ and 
so, by the construction of $\vp_0$, we also have $$\vp_0(x)\asymp \vp_0(N)=\sigma_N, \quad 
N-1\le x\le N+1.$$ Taking into account that $\rho_{N-1}\le|z|
\le\rho_{N+1}$, we get
\[
 \|k^B_z\|^2 \asymp\|k^B_{\rho_N}\|^2 \asymp\sigma_N
 \asymp \vp(|z|).
 \qedhere
\]
\end{proof}
 
It should be noted that Theorem \ref{C-upper-growth} works in a broader context 
assuming less
``tangentiality''. Indeed,
it is clear from the proof that the hypothesis $x_n\downarrow 0$
can be weakened to
\bea\label{WeakCond}
 \sup_{n\ge 1}\frac{x_{n+1}}{x_n}<2,
\eea
since in this case we still have
$x_n\theta_n^2 \ll \theta_n$, $2^{2N}\sum_{k\ge N+1}x_k\theta_k^2
\lesssim x_N \ll \sigma_N$ and $\sigma_N\le \sigma_{N+1}
=\sigma_N+x_{N+1}\le \sigma_N+2x_N\le 2\sigma_N$.

We would now like to consider the sharpness of the growth in Theorem 
\ref{C-upper-growth}.

\begin{theorem} \label{thm4.2}
Suppose $B$ is a Blaschke product whose zeros satisfy the conditions in 
Theorem \ref{C-upper-growth}. Then for every $\varepsilon > 0$ there exists an $f \in (B H^2)^{\perp}$ such that 
\begin{equation} \label{desired-rho-z-9}
|f(z)| \gtrsim \sqrt{\frac{\vp(|z|)}{\log^{1+\eps} \vp(|z|)}}, \quad z \in \Gamma_{\alpha, 1}.
\end{equation}
\end{theorem}

\begin{proof}
Functions in $(B H^2)^{\perp}$ behave rather nicely if the sequence $\Lambda$ is 
interpolating. 
To see this, recall that $x_n \downarrow 0$  and so
$$
 \varlimsup_{k \to \infty} \frac{x_{k+1}}{x_k}\le 1.
$$
Hence
\[ 
 |b_{r_k}(r_{k+1})|= \frac{x_{k}2^{-2k}-x_{k+1}2^{-2(k+1)}}
  {x_{k}2^{-2k}+x_{k+1}2^{-2(k+1)}}
 =\frac{1-\frac{1}{4}\frac{x_{k+1}}{x_{k}}}{1+\frac{1}{4}\frac{x_{k+1}}{x_{k}}}
 \ge 1 - \frac{1}{4}=\frac{3}{4} \mbox{ (asymptotically)}.
\]
Thus the sequence of moduli is pseudo-hyperbolically separated which implies that
the sequence of moduli  is interpolating -- as  will be the one spread out by
the arguments i.e., $\Lambda$.

Now, since $\Lambda$ is an interpolating sequence, we also know that
 the normalized reproducing kernels 
\[
 K_n:=\frac{k_{\lambda_n}}{\|k_{\lambda_n}\|}=\frac{\sqrt{1-|\lambda_n|^2}}
 {1-\overline{\lambda_n}z},\quad n\in\N,
\]
form an unconditional basis for $(B H^2)^{\perp}$. This is essentially
a result by Shapiro and Shields \cite{SS}, see also \cite[Section
3]{Nik2} and in particular \cite[Exercice C3.3.3(c)]{Nik2}. 
Hence for every $f\in (B H^2)^{\perp}$, there is a sequence $\alpha := (\alpha_n)_{n \geq 1} \in \ell^2$ such that
\begin{equation} \label{f-alpha-defn}
 f_{\alpha}(z):=\sum_{n \geq 1} \alpha_n \frac{k_{\lambda_n}(z)}{\|k_{\lambda_n}\|}
 =\sum_{n \geq 1}\alpha_n\frac{\sqrt{1-r_n^2}}{1-r_ne^{-i\theta_n}z}.
\end{equation}
We will examine this series for $z=r\in [0,1)$ (it could be necessary at some
point to require $r\ge r_0>0$).
In what follows we will assume that $\alpha_n>0$.
Note that the argument $1-e^{-i\theta_n}rr_n$  is positive
(this is $\gamma_n$ in Figure 1).

\begin{figure}
\begin{picture}(0,0)%
\includegraphics[width=2in]{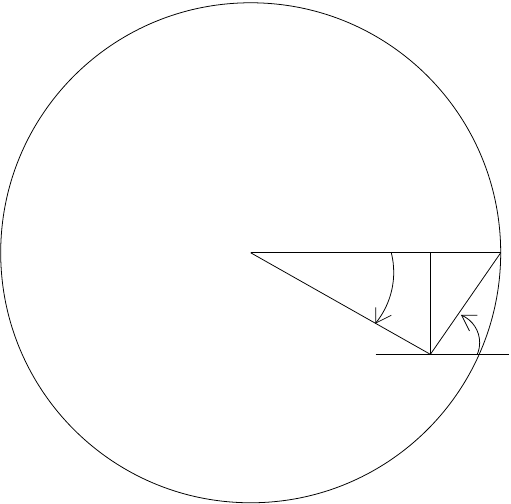}%
\end{picture}%
\setlength{\unitlength}{1973sp}%
\begingroup\makeatletter\ifx\SetFigFont\undefined%
\gdef\SetFigFont#1#2#3#4#5{%
  \reset@font\fontsize{#1}{#2pt}%
  \fontfamily{#3}\fontseries{#4}\fontshape{#5}%
  \selectfont}%
\fi\endgroup%
\begin{picture}(4895,4814)(2993,-6368)
\put(6200,-4300){\makebox(0,0)[lb]{\smash{{\SetFigFont{7}{8.4}{\familydefault}{\mddefault}{\updefault}{\color[rgb]{0,0,0}$-\theta_n$}%
}}}}
\put(7800,-4100){\makebox(0,0)[lb]{\smash{{\SetFigFont{7}{8.4}{\familydefault}{\mddefault}{\updefault}{\color[rgb]{0,0,0}$1$}%
}}}}
\put(7200,-4820){\makebox(0,0)[lb]{\smash{{\SetFigFont{6}{7.2}{\rmdefault}{\mddefault}{\updefault}{\color[rgb]{0,0,0}$\gamma_n$}%
}}}}
\put(6900,-5200){\makebox(0,0)[lb]{\smash{{\SetFigFont{7}{8.4}{\familydefault}{\mddefault}{\updefault}{\color[rgb]{0,0,0}$rr_ne^{-i\theta_n}$}%
}}}}
\end{picture}%
\caption{angles}
\end{figure}

Fix $\rho_N=1-2^{-N}$. Then
\bea\label{decompfalpha}
 f_{\alpha}(\rho_N) 
  = \sum_{n=1}^N \alpha_n 
   \frac{\sqrt{1-r_n^2}}{1-\rho_Nr_ne^{-i\theta_n}}
 +\sum_{n>N} \alpha_n 
   \frac{\sqrt{1-r_n^2}}{1-\rho_Nr_ne^{-i\theta_n}}.
\eea
Let us show that the second term is bounded by
a constant.
By definition $1-r_n=x_n\theta_n^2=x_n2^{-2n}$, and from 
\eqref{eq1p4} $|e^{i\theta_n}-\rho_Nr_n|\asymp \theta_n+(1-\rho_N)
\asymp 1-\rho_N$ for $n\ge N$. 
{
In particular,
\[
\left |\sum_{n>N} \alpha_n 
   \frac{\sqrt{1-r_n^2}}{1-\rho_Nr_ne^{-i\theta_n}}\right|
 \le  \sum_{n>N} \alpha_n\frac{\sqrt{1-r_n^2}}{|e^{i\theta_n}-\rho_Nr_n|}
 \asymp \sum_{n>N} \alpha_n\frac{\sqrt{x_n}\theta_n}{1-\rho_N}
 =2^N\sum_{n>N} \alpha_n\sqrt{x_n}\frac{1}{2^n}.
\]
Now since the terms $\alpha_n\sqrt{x_n}$ are bounded, the last expression is uniformly
bounded in $N$ by a positive constant $M$.
}

Consider the first sum in \eqref{decompfalpha}.
We will show that for $1\le n\le N$ the argument
of $1-e^{-i\theta_n}\rho_Nr_n$ is uniformly close to
$\pi/2$ (or at least from a certain $n_0$ on), meaning that
$1-e^{-i\theta_n}\rho_Nr_n$ points in a direction uniformly close to
the positive imaginary axis.
To this end set
$\gamma_n=\arg(1-\rho_Nr_ne^{-i\theta_n}),$ then
\beqa
 \tan\gamma_n&=&\frac{r_n\rho_N\sin\theta_n}{1-r_n\rho_N\cos\theta_n}
 \simeq \frac{\theta_n}{1-(1-x_n\theta_n^2)(1-\theta_N)(1-\theta_n^2/2+o(\theta_n^2))}\\
 &=&\frac{\theta_n}{x_n\theta_n^2+\theta_N+\theta_n^2/2+o(\theta_n^2)}
  \asymp\frac{\theta_n}{\theta_n^2+\theta_N} 
  \asymp \left\{
 \begin{array}{ll}
 \frac{\dst 1}{\dst\theta_n} &\text{if }n\le N/2\\
 \frac{\dst\theta_n}{\dst\theta_N} &\text{if }N/2<n\le N. 
 \end{array}
 \right.\\
 &=& \left\{\begin{array}{ll}
 2^n &\text{if }n\le N/2\\
 2^{N-n} &\text{if }N/2<n\le N. 
 \end{array}
 \right.\\
 &\ge& 1.
\eeqa
Hence the argument of $1-\rho_Nr_ne^{-i\theta_n}$ is uniformly bounded away
from zero and less than $\pi/2$ so that
\[
 1\ge \sin \arg(1-\rho_Nr_ne^{-i\theta_n})\ge \eta>0.
\]
In particular, for $1\le n\le N$,
\[
\left |\Im \frac{1}{1-\rho_Nr_ne^{-i\theta_n}}\right|\asymp
 \frac{1}{|1-\rho_Nr_ne^{-i\theta_n}|} \asymp \frac{1}{\theta_n+(1-\rho_N)}
 \asymp\frac{1}{\theta_n}.
\]
This implies that
\beqa
 |f_{\alpha}(\rho_N)|
 &=&\left|\sum_{n\ge 1} \alpha_n 
   \frac{\sqrt{1-r_n^2}}{1-\rho_Nr_ne^{-i\theta_n}}\right|
 \ge \left|\sum_{n=1}^N \alpha_n 
   \frac{\sqrt{1-r_n^2}}{1-\rho_Nr_ne^{-i\theta_n}}\right|
  - \left|\sum_{n>N} \alpha_n 
   \frac{\sqrt{1-r_n^2}}{1-\rho_Nr_ne^{-i\theta_n}}\right|\\
 &\ge&\left|\sum_{n=1}^N \alpha_n 
   \frac{\sqrt{1-r_n^2}}{1-\rho_Nr_ne^{-i\theta_n}}\right|
  - M
  \asymp \sum_{n=1}^N \alpha_n \sqrt{1-r_n^2}\times
  \left|\Im \frac{1}{1-\rho_Nr_ne^{-i\theta_n}}\right|-M\\
 &\asymp&\sum_{n=1}^N \alpha_n\frac{\sqrt{x_n}\theta_n}{\theta_n}-M\\
 &=&\sum_{n=1}^N \alpha_n\sqrt{x_n}-M.
\eeqa

As we will see, for a specific choice of sequence $(\alpha_n)_{n \geq 1}$,  the sum
$\sum_{n=1}^N \alpha_n\sqrt{x_n}$ will tend to infinity implying that
in such a situation the constant $M$ is negligible and
\begin{equation} \label{replace-rho-z}
 |f_{\alpha}(\rho_N)|\gtrsim\sum_{n=1}^N \alpha_n\sqrt{x_n}.
\end{equation}
Let us discuss the following choice for $\alpha_n$:
\[
 \alpha_n:=\sqrt{\frac{x_n}{\sigma_n\log^{1+\eps}\sigma_n}}.
\]
We need to show two things (i) we get the desired lower estimate in the statement of the theorem; and (ii)  $(\alpha_n)_{n \geq 1} \in \ell^2$.
Let us begin with the lower estimate. Observe that $\sigma_{N}$ is
increasing and so
\beqa
 \sum_{n=1}^N\alpha_n\sqrt{x_n}
 &=&\sum_{n=1}^N\frac{\sqrt{x_n}}{\sqrt{\sigma_n\log^{1+\eps}\sigma_n}} 
  \sqrt{x_n}
 =\sum_{n=1}^N\frac{x_n}{\sqrt{\sigma_n\log^{1+\eps}\sigma_n}}
 \ge \frac{1}{\sqrt{\sigma_N\log^{1+\eps}\sigma_N}}\sum_{n=1}^Nx_n\\
 &=&\frac{\sigma_N}
 {\sqrt{\sigma_N\log^{1+\eps}\sigma_N}}\\
 &=&\sqrt{\frac{\sigma_N}{\log^{1+\eps}\sigma_N}}.
\eeqa
This proves that 
$$|f(\rho_N)| \gtrsim \sqrt{\frac{\sigma_N}{\log^{1+\eps}\sigma_N}}.$$
To get the desired inequality in \eqref{desired-rho-z-9} (i.e., replace $\rho_N$ with $z \in \Gamma_{\alpha, 1}$) go back to the argument which proved the inequality in \eqref{replace-rho-z} and use the argument used to prove \eqref{QQQ}.

To show that $(\alpha_{n})_{n \geq 1} \in \ell^2$, observe that
\[
 \sum_{n=1}^N\alpha_n^2=\sum_{n=1}^N
  \frac{x_n} {{\sigma_n\log^{1+\eps}\sigma_n}}
 =\sum_{n=1}^N\frac{\sigma_n-\sigma_{n-1}}{{\sigma_n\log^{1+\eps}\sigma_n}},
\]
where we set $\sigma_0=0$.

Since $-g'$ is decreasing, where 
\[
 g(t)=\frac{1}{\log^{\eps}(t)}, \quad t \in [1, \infty),
\]
and using the fact that $(\sigma_n)_{n \geq 1}$ is increasing,
is it possible to show that
\bea\label{estimsigma'}
 \frac{\sigma_n-\sigma_{n-1}}{{\sigma_n\log^{1+\eps}\sigma_n}}
 \le \frac{1}{\eps}\left(\frac{1}{\log^{\eps}\sigma_{n-1}}-
  \frac{1}{\log^{\eps}\sigma_{n}}\right).
\eea

Hence
\beqa
 \sum_{n=2}^N\alpha_n^2 
 &=&\sum_{n=2}^N\frac{\sigma_n-\sigma_{n-1}}{\sigma_n\log^{1+\eps}\sigma_n}
 \le \frac{1}{\eps}\sum_{n=2}^N
 \left(\frac{1}{\log^{\eps}\sigma_{n-1}}-\frac{1}{\log^{\eps}\sigma_{n}}\right)
 =\frac{1}{\eps}\left( \frac{1}{\log^{\eps}\sigma_1}-
 \frac{1}{\log^{\eps}\sigma_N}\right)\\
 &\le& \frac{1}{\eps \log^{\eps}\sigma_1}  \qedhere
\eeqa
\end{proof}


\begin{remark} \label{R-general}
If one looks closely at the proof of Theorems \ref{C-upper-growth} and \ref{thm4.2} one can show that given any concave growth function $\vp_0$ one 
can create a Blaschke product $B$ so that the functions in $(B H^2)^{\perp}$ have growth rates controlled by the associated $\vp$ (upper control as in Theorem \ref{C-upper-growth} and lower control as in Theorem \ref{thm4.2}). 
\end{remark}




Without going into cumbersome technical details, here is another
remark on the optimality of Theorem \ref{thm4.2}. 
We are interested in the following question: for which sequences 
$\eps_n \downarrow 0$
does there exist a sequence $(\alpha_n)_{n \geq 1}\in \ell^2$ such that
\bea\label{optimality}
 \sum_{n=1}^N\alpha_n\sqrt{x_n}=\eps_N\sigma_N\ ?
\eea
For example, when $x_n \equiv 1$ (Theorem \ref{thm4.2} is still valid in
this setting) 
we have $\sigma_{N} = N$ and 
 the question becomes: for which sequences
$\eps_n \downarrow 0$ does there exist a sequence $(\alpha_n)_{n \geq 1 }\in \ell^2$ such that
\bea\label{optim1}
 \sum_{n=1}^N\alpha_n=\eps_N\sqrt{N}\ ?
\eea
It is possible to show that, in this case, we can take $\alpha_n$ to be 
\[
 \alpha_n= \eps_n\sqrt{n}-\eps_{n-1}\sqrt{n-1},
\]
which, since $(\alpha_n)_{n \geq 1}\in \ell^2$, yields 
\[
 \sum_n\frac{\eps_n^2}{n}=\sum_n\frac{\eps_n}{\sigma_n} <\infty.
\]
So, for instance,
if we were to choose $\eps_n=1/\log^{\alpha}n$, then we would need
$\alpha>1/2$ which is, in a sense, optimal in view of the preceding corollary.

A crucial point in this discussion is the fact that $(\eps_n)_{n \geq 1}$ is a \emph{decreasing} sequence.

\noindent {\bf Second class of examples:}
In the preceding class of examples from \eqref{defseq}, we slowed down the growth of functions in
$(B H^2)^{\perp}$ by controlling the ``tangentiality'' of the sequence (given by the
speed of convergence to zero of $x_n$). Our second class of examples are 
of the following type:
\begin{equation} \label{second}
 \lambda_n=r_ne^{i\theta_n}, \quad 0 < \theta_n < 1, \quad 1-r_n=\theta_n^2, \quad \sum_{n \geq 1} \theta_n < \infty,
\end{equation}
where $\theta_n$ can be adjusted to control the growth
speed of $(B H^2)^{\perp}$-functions. Asymptotically, this sequence is in the
oricycle $\{z\in\D:|z-1/2|=1/2\}$.
We also note that 
$$\sum_{n \geq 1} (1 - |\lambda_n|) = \sum_{n \geq 1} \theta_{n}^{2} < \infty$$ so indeed $(\lambda_{n})_{n \geq 1}$ is a Blaschke sequence. Moreover, 
\begin{equation} \label{F-cond-A}
\sum_{n \geq 1} \frac{1 - |\lambda_n|}{|1 - \lambda_n|} \asymp \sum_{n \geq 1}   \frac{\theta_{n}^{2}}{\theta_n} = \sum_{n \geq 1} \theta_{n} < \infty
\end{equation}
 and so, by \eqref{Frost},  $\lim_{r \to 1^{-}} B(r) = \eta \in \T$. Still further, we have 
$$\sum_{n \geq 1}  \frac{1 - |\lambda_n|}{|1 - \lambda_n|^2} \asymp \sum_{n \geq 1}\frac{\theta_{n}^{2}}{\theta_{n}^2} =  + \infty$$
so $(\lambda_{n})_{n \geq 1}$ does not satisfy the hypothesis \eqref{AC1} of the Ahern-Clark theorem. Thus we can expect bad behavior of functions from $(B H^2)^{\perp}$.

As in \eqref{Pythag-Lemma}, we have
\beqa
 \frac{1-|\lambda_k|^2}{|1-r\lambda_k|^2}
 \asymp\frac{1-r_k}{(1-r)^2+\theta_k^2}
 =\frac{\theta_k^2}{(1-r)^2+\theta_k^2}
&\asymp& \left\{\begin{array}{ll}
	 1 & \text{if }(1-r)\le \theta_k\\
	\frac{\dst\theta_k^2}{\dst (1-r)^2} & \text{if }(1-r)>\theta_k.
 \end{array}
 \right.
\eeqa
Using again Lemma \ref{KeyLemma}, the splitting gives:
\bea\label{splitting}
 \|k_r^B\|^2\asymp \sum_{k\ge 1}\frac{1-|\lambda_k|^2}{|1-r\lambda_k|^2}
 \asymp  \sum_{\{k:(1-r)\le\theta_k\}} 1 
 +\frac{1}{(1-r)^2}\sum_{\{k:(1-r)> \theta_k\}}{\theta_k^2}.
\eea

\begin{theorem} \label{Oricycle-seq}
Let $(\sigma_N)_{N \geq 1}$ be a sequence of positive numbers strictly increasing to infinity,
and
\bea\label{GrCondSigma}
  \sigma_{N+1}\le 2^{\beta}\sigma_N,\quad N\in \N,
\eea
for some $\beta\in (0,2)$.
Then there exists a sequence $(\theta_k)_{k \geq 1} \in \ell^1$ such that 
\[
 \|k_{\rho_N}^B\|\asymp \sqrt{\sigma_N},
\]
where $B$ is the Blaschke product whose zeros are $\Lambda=(\lambda_k)_{k \geq 1}$
and $\lambda_k=r_ke^{i\theta_k}$, $1-r_k=\theta_k^2$.
\end{theorem}


\begin{proof}
Let $(\sigma_N)_{N \geq 1}$ be as in the theorem and let
\[
 \psi:[0,+\infty)\lra [0,+\infty)
\]
be a continuous increasing function such that 
\bea\label{pointvalues}
 \psi(N)=\sigma_N,\quad N\in\N.
\eea
We could, for example, choose $\psi$ to be the continuous piecewise affine
function defined at the nodes by \eqref{pointvalues}.
Since $\psi$ is continuous and strictly increasing to infinity on $[0,+\infty)$,
it has an inverse function $\psi^{-1}$.
Set
\[
 \theta_k=2^{-\psi^{-1}(k)},\quad k\in \N.
\]
We need to show that $(\theta_n)_{n \geq 1} \in \ell^1$ (in order to satisfy the Frostman condition in \eqref{F-cond-A}) but this will come out of our analysis below. 
Let us consider the first sum in \eqref{splitting} (with $r = \rho_N$):
\[
 \sum_{\{k:(1-\rho_N)\le\theta_k\}} 1 
 =\sum_{\{k:1/2^N\le 1/2^{\psi^{-1}(k)}\}} 1 
 =\sum_{\{k:{\psi^{-1}(k)}\le N\}} 1 
 =\sum_{\{k:k\le {\psi(N)}\}} 1=\psi(N)=\sigma_N .
\]

We have to consider the second sum in \eqref{splitting}:
\beqa
 \frac{1}{(1-\rho_N)^2}\sum_{\{k:(1-\rho_N)> \theta_k\}}{\theta_k^2}
 =2^{2N} \sum_{\{k: {\psi^{-1}(k)}\ge N+1\}} 2^{-2\psi^{-1}(k)}
  =2^{2N} \sum_{\{k \ge {\psi(N+1)}\}} 2^{-2\psi^{-1}(k)}.
\eeqa

%
%
%
Since $\sigma_n=\psi(n)$, equivalently $\psi^{-1}(\sigma_n)=n$,
we have 
\bea\label{estimBlaschke2a}
 \sum_{\{k \ge {\psi(N+1)}\}} 2^{-2\psi^{-1}(k)}
 &=&\sum_{n\ge N+1}\sum_{k=\sigma_n}^{\sigma_{n+1}-1}2^{-2\psi^{-1}(k)}
 \le \sum_{n\ge N+1} (\sigma_{n+1}-\sigma_n)
 2^{-2\psi^{-1}(\sigma_n)}\nn\\
 &\le& \sum_{n\ge N+1}\frac{1}{2^{2n}}\sigma_{n+1}\nn\\
 &\le& 2^{\beta}\sum_{n\ge N+1}\frac{1}{2^{2n}}\sigma_{n}.
\eea
Now, setting $u_n=\sigma_n/2^{2n}$,
we get $v_n=u_{n+1}/u_n\le 2^{\beta-2}<1$, from which
standard arguments give
\bea\label{estimBlaschke2b}
 \sum_{n\ge N+1}\frac{\sigma_{n}}{2^{2n}}
 \lesssim \frac{\sigma_N}{2^{2N}}.
\eea
Hence
\[
 2^{2N}\sum_{\{k \ge {\psi(N+1)}\}} 2^{-2\psi^{-1}(k)}\lesssim \sigma_N
\]
So, 
according to \eqref{splitting},
\beqa
 \sigma_N\le
 \underbrace{\sum_{\{k:(1-r)\le\theta_k\}} 1 
 +\frac{1}{(1-r)^2}\sum_{\{k:(1-r)> \theta_k\}}{\theta_k^2}}_{
 \asymp \|k^B_{\rho_N}\|^2}
 \lesssim \sigma_N+ \sigma_N
\eeqa
which completes the proof.
\end{proof}

\begin{remark}
Note that the Blaschke condition for $\Lambda$ is given by
\[
 \sum_k(1-|\lambda_k|^2)\simeq \sum_k (1-r_k)
 =\sum_k \theta_k^2=\sum_k 2^{-2\psi^{-1}(k)}<\infty.
\]
Combining for instance \eqref{estimBlaschke2a} and
\eqref{estimBlaschke2b} it can be seen that the condition
$0<\beta<2$ (condition \eqref{GrCondSigma}) guarantees that $\Lambda$ is
a Blaschke sequence.
\end{remark}

\begin{example}\label{Example4.32}
Here is a list of examples. 
\begin{enumerate}
\item Let $\sigma_N=2^{N/\alpha}$, $N = 1, 2, \ldots$, where $\alpha>1$
(this is needed for \eqref{GrCondSigma}).
Then, we can choose $\psi(t)=2^{t/\alpha}$. Hence
\[
 \theta_k=2^{-\psi^{-1}(k)}=2^{-\alpha \log k}=\frac{1}{k^{\alpha}}
\]
(logarithms are base 2). Hence, with this choice of arguments,
we get
\[
 \|k_{\rho_N}^B\|\asymp 2^{N/2\alpha}=\frac{1}{(1-\rho_N)^{1/2\alpha}},
\] 
which by similar arguments as given earlier (see the proof of Theorem \ref{C-upper-growth}) can be extended to every
$r\in (0,1)$, i.e., 
$$|f(r)| \lesssim \frac{1}{(1 - r)^{1/2\alpha}}, \quad f \in (B H^2)^{\perp}.$$
We thus obtain all power growths beyond the limiting case $1/2$.

\item Let $\sigma_N={N}^{\alpha}$, $N = 1, 2, \ldots$, where $\alpha>0$.
Then we can choose $\psi(t)={t}^{\alpha}$. Hence
\[
 \theta_k=2^{-\psi^{-1}(k)}=2^{-k^{1/\alpha}},
\]
and, with this choice of arguments,
we get
\[
 \|k_{\rho_N}^B\|\asymp N^{\alpha/2}=\left(\log\frac{1}{1-\rho_N}\right)^{\alpha/2}.
\] 
Thus as in (1) we get
$$|f(r)| \lesssim \left(\log\frac{1}{1-r}\right)^{\alpha/2},  \quad f \in (B H^2)^{\perp}.$$
In the special case $\alpha=1$ we obtain logarithmic growth.

\item  Let $\sigma_N=\log^2{N}$, $N\ge 2$. 
Then we can choose $\psi(t)=\log^2{t}$. Hence
\[
 \theta_k=2^{-\psi^{-1}(k)}=2^{-2^{\sqrt{k}}}.
\]
With this choice of arguments,
we get, for large enough $N$,
\[
 \|k_{\rho_N}^B\|\asymp {\log N}={\log \log\frac{1}{1-\rho_N}},
\] 
and so
$$|f(r)| \lesssim \log \log \frac{1}{1 - r}, \quad f \in (B H^2)^{\perp}.$$
\end{enumerate}
\end{example}

\section{A general growth result for $(B H^2)^{\perp}$}

It turns out 
that growth results can be phrased in terms of a more general result. 
In fact our first class of examples can be deduced from such a
general result (see Remark \ref{RemExam}).

We will start by introducing a growth parameter associated
with a Blaschke sequence $\Lambda=(\lambda_n)_{n \geq 1} \subset \D$
and a boundary point $\zeta\in \T$.
Let us again set
\[
 \rho_N:=1-\frac{1}{2^N}, \quad N\in\N.
\]
For every $N\in\N$ and $n\in\Z$, set 
\begin{equation} \label{GNn}
 \Gamma^{N,\zeta}_n:=\left\{z\in\D:\frac{1-|z|^2}{|\zeta-\rho_Nz|^2}
 \in \left[{\frac{1}{2^{n+1}}},\frac{1}{{2^{n}}}\right)\right\}.
\end{equation}
This is a kind of pseudo-hyperbolic annulus (see Figure 2). A routine
computation shows that
$$\frac{\dst 1-|z|^2}{\dst |\zeta-\rho z|^2}=c
\Longleftrightarrow
 \left|z-\frac{c\rho}{1+c\rho^2}\zeta\right|^2=\frac{1-c(1-\rho^2)}{(1+c\rho^2)^2}.
$$
\begin{figure}
\includegraphics[width = 2in]{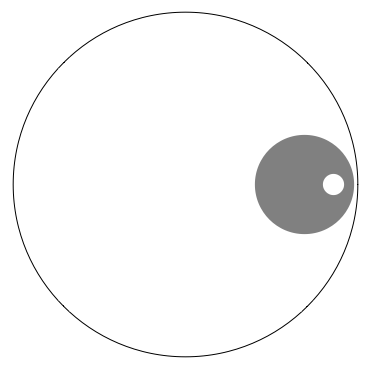}
\caption{An example of a domain $\Gamma_{n}^{N, 1}$.}
\end{figure} 
From here observe that necessarily 
$c\le \frac{\dst 1}{\dst 1-\rho^2}$ which means that $\Gamma^{N,\zeta}_n$
is empty when $$\frac{1}{2^{n+1}}\ge \frac{\dst 1}{\dst 1-\rho_N^2}
\ge \frac{\dst 1}{\dst 2(1-\rho_N)}=  2^{N-1}.$$
Thus we assume that $n\ge -N$. 

For simplicity, we will assume from now on that
$\zeta=1$ and set
\[
 \Gamma^N_n:=\Gamma^{N,1}_n.
\]
Define
\[
 \alpha_{N, n} :=\#(\Lambda\cap \Gamma^N_n)
\]
(the number of points in $\Lambda\cap \Gamma^N_n$)
along with the growth parameter
\[
 \sigma_N^{\Lambda}:=\sum_{n\in\Z}\frac{\alpha_{N, n}}{2^n}
 =\sum_{n\ge -N}\frac{\alpha_{N, n}}{2^n}.
\]

 For each $\lambda \in \Lambda \cap \Gamma^{N}_{n}$ we have, by definition (see \eqref{GNn}), 
 $$\frac{1}{2^{n}} \asymp \frac{1 - |\lambda|^2}{|1 - \rho_N \lambda|^2}$$ and so, since there are $\alpha_{N, n}$ points in $\Lambda \cap \Gamma^{N}_{n}$, we have 
 $$\sum_{n\ge -N}\frac{1}{2^n}\#(\Lambda\cap\Gamma^N_n) \asymp \sum_{n\ge -N}\sum_{\lambda\in\Lambda\cap\Gamma^N_n}
 \frac{1-|\lambda|^2}{|1-\rho_N\lambda|^2}.$$
  But since $(\Gamma^N_n)_{n\ge -N}$ is
a partition of $\D$ (see Figure 3)
\begin{figure}
\includegraphics[width = 2in]{PsHy-region2.png}
\includegraphics[width = 2in]{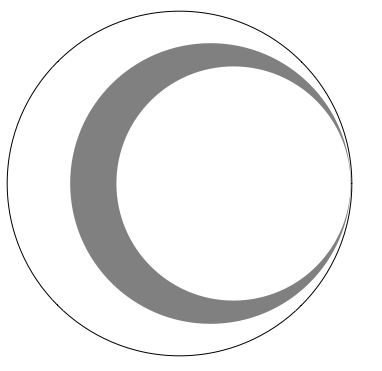}
\caption{The domains $\Gamma^{N}_{n}, -N \leq n $ cover $\D$. }
\end{figure} 
we get 
$$\sum_{n\ge -N}\sum_{\lambda\in\Lambda\cap\Gamma^N_n}
 \frac{1-|\lambda|^2}{|1-\rho_N\lambda|^2}
=\sum_{n \geq 1}
 \frac{1-|\lambda_n|^2}{|1-\rho_N\lambda_n|^2}.
$$
Putting this all together we arrive at
\begin{equation} \label{sigmaN}
\sigma_{N}^{\Lambda} \asymp \sum_{n \geq 1}
 \frac{1-|\lambda_n|^2}{|1-\rho_N\lambda_n|^2}.
\end{equation}

Combine \eqref{sigmaN} with Lemma \ref{KeyLemma}  to get the two-sided estimate 
\begin{equation} \label{sigmaN-kr}
\sigma_{N}^{\Lambda} \asymp \|k^B_{\rho_{N}}\|^2.
\end{equation}

Note that if the zeros $(\lambda_{n})_{n \geq 1}$ of $B$ satisfy
the Ahern-Clark condition \eqref{AC1}  then, by Theorem \ref{AC-paper},
the sequence $(\|k^B_{\rho_N}\|)_{N \geq 1}$ is uniformly bounded and, by \eqref{sigmaN-kr}, so is  $(\sigma_N^{\Lambda})_{N \geq 1}$.

To discuss the case when 
$(\sigma_N^{\Lambda})_{N \geq 1}$ is unbounded, 
we will impose the mild regularity condition
\bea\label{regcond}
 0<m:=\inf_N\frac{\sigma^{\Lambda}_{N+1}}{\sigma^{\Lambda}_N}\le M
 := \sup_N\frac{\sigma^{\Lambda}_{N+1}}
 {\sigma^{\Lambda}_N}<\infty.
\eea
In Section \ref{S3nn}, this condition was automatically satisfied by $\sigma_N
=\sum_{k=1}^Nx_k$.

Let us associate with $\sigma^{\Lambda}_N$ the functions $\vp_0$ and $\vp$ as
in Theorem \ref{C-upper-growth}.
Then,
%
from \eqref{sigmaN-kr} we deduce the following result in the
same way as Theorem \ref{C-upper-growth}.

\begin{theorem}\label{thm3.1}
Let $\Lambda=(\lambda_n)_{n \geq 1}\subset \D$ be a Blaschke sequence
with associated growth sequence $\sigma^{\Lambda}=
(\sigma^{\Lambda}_{N})_{N \geq 1}$ at $\zeta=1$ satisfying  \eqref{regcond}
and $B$ the Blaschke product with zeros $\Lambda$. Then 
\[
 \|k^B_z\|\asymp \sqrt{\vp(|z|)}, \quad z \in \Gamma_{\alpha, 1}.
\]
Consequently, every  $f\in (B H^2)^{\perp}$ satisfies
\[
 |f(z)|=|\langle f,k_z\rangle| \lesssim \sqrt{\vp(|z|)}, \quad z \in \Gamma_{\alpha, 1}.
\] 
\end{theorem}

\begin{remark}\label{RemExam}
It turns out that for the sequences discussed in Theorem \ref{C-upper-growth}  
we have 
$$\sigma_{N}^{\Lambda} \asymp \sigma_{N} = \sum_{k = 1}^{N} x_{k}.$$
The details are somewhat cumbersome so we will not give them here. 
\end{remark}

\section{A final remark on unconditional bases}\label{section5}

Since a central piece of our discussion was the behavior of the reproducing 
kernels $k^{B}_{\rho_N}$, one could ask whether or not $(k^B_{\rho_N})_{N \geq 1}$ forms
an unconditional bases (or sequence)  for $(B H^2)^{\perp}$.

To this end, let $x_n=k^B_{\rho_n}/\|k^B_{\rho_n}\|$ 
and $G=(\langle x_n,x_k\rangle)_{n,k}$ be the associated Gram matrix.
Suppose that $(x_n)_{n \geq 1}$ were an unconditional basis 
(or sequence) for $(B H^2)^{\perp}$. 
In this case, it is well known (see e.g.\ \cite[Exercise C3.3.1(d)]{Nik2}) 
 that $G$ represents an isomorphism
from $\ell^2$ onto $\ell^2$. 
It follows from the unconditionality of $(x_n)_{n \geq 1}$ that
every $f\in (B H^2)^{\perp}$ (or every $f$ in the span of $(x_n)_{n \geq 1}$) can be written as
\[
 f=f_{\alpha}:=\sum_{n\ge 1} \alpha_n x_n, \quad \alpha = (\alpha_{n})_{n \geq 1} \in \ell^2,
\]
with $\|f_{\alpha}\|^2\asymp\sum_{n\ge 1}|\alpha_n|^2<\infty$.
As before we want to estimate $f=f_{\alpha}$ at $\rho_N$. Indeed, 
\[
 f_{\alpha}(\rho_N)=\sum_{n\ge 1}\alpha_n\frac{k^B_{\rho_n}(\rho_N)}{\|k^B_{\rho_n}\|}
 =\|k^B_{\rho_N}\|\sum_{n\ge 1}\alpha_n\frac{\langle k^B_{\rho_n},k^B_{\rho_N}
   \rangle}{\|k^B_{\rho_n}\| \|k^B_{\rho_N}\|}
 =\|k^B_{\rho_N}\| (G \alpha)_{N}.
\]
Again we observe that for every $\alpha \in \ell^2$, we have
\[
  f_{\alpha}(\rho_N) = \|k^B_{\rho_N}\| (G \alpha)_{N}
\]
where $G \alpha \in \ell^2 $,
and  for every $\ell^2$-sequence $\beta$ we
could find an $f\in (B H^2)^{\perp}$ such that $$\frac{f(\rho_N)}{\|k^B_{\rho_N}\|}=\beta_N.$$

However, recall from Remark \ref{R-general} 
that for $\eps>0$ there is a function $f_{\alpha}$ with
\[
 |f_{\alpha}(\rho_N)|\gtrsim \sqrt{\frac{\sigma_N}{\log^{1+\eps}\sigma_N}}
\]
(we refer to that remark for notation).
Since by Theorem \ref{C-upper-growth} we have $\|k_{\rho_N}^B\|\asymp \sigma_N$,
we would thus have
\[
 \beta_N:=\frac{|f_{\alpha}(\rho_N)|}{\|k^B_{\rho_N}\|}
 \asymp \frac{|f_{\alpha}(\rho_N)|}{\sqrt{\sigma_N}}
 \gtrsim\frac{1}{\log^{(1+\eps)/2}\sigma_N}.
\]
However, for instance, choosing 
$x_n=1/n$ yields $\sigma_N\simeq \log N$, in which case
$(1/\log^{(1+\eps)/2}\sigma_N)_{N \geq 1}$ is obviously not in $\ell^2$.
(In fact, a closer look at the proof of Theorem \ref{thm4.2} shows that
one can also choose $x_n=1$ to get a sequence $(\beta_N)_{N \geq 1} \not \in \ell^2$.) 
As a result, we can conclude that in the above examples $(k^B_{\rho_N})_{N \geq 1}$
cannot be an unconditional basis for $(B H^2)^{\perp}$ 
(nor an unconditional sequence since the functions in 
Theorem \ref{thm4.2}
were constructed using the reproducing kernels, so
they belong the space spanned by $(x_n)_{n \geq 1}$).

It should be noted that the problem of deciding whether or not a sequence of
reproducing kernels forms an unconditional basis (or sequence) for a model space is a
difficult problem related to the  Carleson condition and the invertibility of
Toeplitz operators. We do not want to go into details here, but the situation 
becomes even more difficult in our context where $\varlimsup_N|B(\rho_N)|=1$. See
\cite[Chapter D4]{Nik2} for more about this.

\bibliography{AC}

\end{document}